\documentclass{amsart}[11pt]
\usepackage{amssymb}
\DeclareMathSymbol{\twoheadrightarrow} {\mathrel}{AMSa}{"10}

\def\Q{{\mathbf Q}}

        \def\CC{{\mathfrak C}}
        
\def\Z{{\mathbf Z}}
\def\C{{\mathbf C}}

\def\H{{\mathrm H}}
\def\ST{{\mathbf S}}

\def\Sn{{\mathbf S}_n}
\def\An{{\mathbf A}_n}

\def\Gal{\mathrm{Gal}}

\def\Perm{\mathrm{Perm}}

\def\Lie{\mathrm{Lie}}

\def\End{\mathrm{End}}

\def\Aut{\mathrm{Aut}}
\def\Hom{\mathrm{Hom}}

    \def\RR{\mathfrak{R}}
\def\I{\mathrm{Id}}

\def\GL{\mathrm{GL}}

\def\SL{\mathrm{SL}}

                \def\sL{\mathfrak{sl}}

                                                \def\su{\mathfrak{su}}
                                                \def\u{\mathfrak{u}}
        
        \def\K_a{\bar{K}}

\def\E{\mathrm{E}}
\def\dim{\mathrm{dim}}

\def\g{{\mathfrak g}}
\def\c{{\mathfrak c}}

\def\G{{\mathbf G}}
\def\GG{{\mathfrak G}}
                          \def\k{{\mathfrak k}}

                           \def\Hdg{\mathrm{Hdg}}

                      \def\hdg{\mathrm{hdg}}
                            \def\Tr{\mathrm{Tr}}
                            \def\U{\mathrm{U}}
                            \def\f{{\mathfrak f}}
\newtheorem{thm}{Theorem}[section]
\newtheorem{lem}[thm]{Lemma}
\newtheorem{cor}[thm]{Corollary}
\newtheorem{prop}[thm]{Proposition}
\theoremstyle{definition}

\newtheorem{rem}[thm]{Remark}

\newtheorem{sect}[thm]{}

\title[Hodge groups of certain superelliptic jacobians]{Hodge groups of certain superelliptic jacobians}
\author{Jiangwei Xue}

\address{Department of Mathematics, Pennsylvania
State University, University Park, PA 16802, USA}

\email{xue\_j\char`\@math.psu.edu}

\author{Yuri G. Zarhin}
\address{Department of Mathematics, Pennsylvania
State University, University Park, PA 16802, USA}
\email{zarhin\char`\@math.psu.edu}

\begin{document}
\maketitle

Throughout this paper $K$ is a field of characteristic zero, $\K_a$ its
algebraic closure and $\Gal(K)=\Aut(\K_a/K)$ the absolute Galois group of $K$.
 If $X$ is an abelian variety over
$\K_a$ then we write $\End(X)$ for the ring of all its
$\K_a$-endomorphisms and $\End^{0}(X)$ for the corresponding
$\Q$-algebra $\End(X)\otimes\Q$; the notation $1_X$ stands for the
identity automorphism of $X$.

Let $f(x)\in K[x]$ be a polynomial of degree $n\ge 3$ with
coefficients in $K$ and without multiple roots, $\RR_f\subset \K_a$
the ($n$-element) set of roots of $f$ and $K(\RR_f)\subset \K_a$
the splitting field of $f$. We write $\Gal(f)=\Gal(f/K)$ for the
Galois group $\Gal(K(\RR_f)/K)$ of $f$; it permutes roots of $f$
and may be viewed as a certain permutation group of $\RR_f$, i.e.,
as as a subgroup of the group $\Perm(\RR_f)\cong\Sn$ of
permutation of $\RR_f$. ($\Gal(f)$ is transitive if and only if
$f$ is irreducible.)

 Suppose that  $p$ is a prime that does {\sl not} divide $n$ and
 a positive integer $q=p^r$ is a power of $p$.
 We write $C_{f,q}$
for the superelliptic $K$-curve $y^q=f(x)$ and $J(C_{f,q})$ for
its jacobian. Clearly, $J(C_{f,q})$ is an abelian variety that is
defined over $K$ and
$$\dim(J(C_{f,q}))=\frac{(n-1)(q-1)}{2}.$$
Assume that $K$ contains a primitive $q$th root of unity $\zeta_q$. In a series
of papers \cite{ZarhinMRL,ZarhinCrelle,ZarhinCamb,ZarhinM}, one of the authors
(Y.Z.) discussed the structure of $\End^0(J(C_{f,q}))$, assuming that $n\ge 5$
and the Galois group $\Gal(f)$ of $f(x)$ over $K$ is, at least, doubly
transitive. In particular, he proved that if $n \ge 5$ and $\Gal(f)$ coincides
either with full symmetric group $\Sn$ or with  alternating group $\An$ then
$\End^0(J(C_{f,q}))$ is (canonically) isomorphic to a product
$\prod_{i=1}^r\Q(\zeta_{p^i})$ of cyclotomic fields. (If $q=p$ then we proved
that $\End(J(C_{f,p}))=\Z[\zeta_p]$.) More precisely, if $q\ne p$ then the map
$(x,y) \to (x,y^p)$ defines the map of curves $C_{f,q}\to C_{f,q/p}$, which
induces (by Albanese functoriality) the surjective homomorphism $J(C_{f,q})\to
J(C_{f,q/p})$
 of abelian varieties over $K$; we write $J^{(f,q)}$ for the identity component of its kernel.
 (If $q=p$ then we put $J^{(f,q)}=J(C_{f,p})$.) One may check \cite{ZarhinM}  that
  $J(C_{f,q})$ is $K$-isogenous to the product $\prod_{i=1}^r J^{(f,p^i)}$ and the automorphism
  $\delta_q: (x,y)\mapsto (x,\zeta_q y)$ of $C_{f,q}$ gives rise to an embedding
  $\Z[\zeta_q]\hookrightarrow \End(J^{(f,q)}), \ \zeta_q \mapsto \delta_q$. What was actually proved in \cite{ZarhinM,ZarhinMZ2}
  is that
$$\Z[\zeta_q]\cong \Z[\delta_q]= \End(J^{(f,q)})$$ if $\Gal(f)=\Sn$, $p\ge 3$ and
$n\ge 4$ or $\Gal(f)=\An$ and $n\ge 5$.

Let us assume that $K\subset\C$ and let the field $\K_a$ be the algebraic
closure of $K$ in $\C$. This allows us to consider $J(C_{f,q})$ and $J^{(f,q)}$
 as complex abelian varieties.
Our goal is to study the (reductive $\Q$-algebraic connected) Hodge group
$\Hdg=\Hdg(J^{(f,q)})$ of $J^{(f,q)}$. Notice that when $q=2$ (i.e., in the
hyperelliptic case) this group was completely determined in \cite{ZarhinMMJ}
(when $f(x)$ has ``large" Galois group).  When $q>2$ we determined in our
previous paper \cite{XueZ} the center of $\Hdg(J^{(f,q)})$, also assuming that
the Galois group of $f(x)$ is ``large".

Let us assume that $q>2$. In order to describe our results let us recall that
the jacobian $J(C_{f,q})$ carries the canonical principal polarization that is
invariant under all automorphisms (induced by automorphisms) of $C_{f,q}$. This
implies that the induced polarization on the abelian subvariety $J^{(f,q)}$ is
$\delta_q$-invariant. This polarization gives rise to the $\delta_q$-invariant
nondegenerate alternating $\Q$-bilinear form
$$\psi_q: \H_1(J^{(f,q)},\Q) \times \H_1(J^{(f,q)},\Q) \to \Q$$
on the first rational homology group of the complex abelian variety
$J^{(f,q)}$. On the other hand, $\H_1(J^{(f,q)},\Q)$ carries the natural
structure of $\Q[\delta_q]\cong \Q(\zeta_q)$-vector space. The
$\delta_q$-invariance of $\psi_q$ implies that
$$\psi_q( e x,y)= \psi_q(x,\bar{e}y) \quad \forall e\in \Q(\zeta_q); \ x,y \in
\H_1(J^{(f,q)},\Q).$$ Here $e \mapsto \bar{e}$ stands for the complex
conjugation map. Let
$$\Q(\zeta_q)^{+}=\{e\in \Q(\zeta_q) \mid \bar{e}=e\}$$
be the maximal totally real subfield of the cyclotomic CM field $\Q(\zeta_q)$
and let
$$\Q(\zeta_q)_{-}=\{e\in \Q(\zeta_q) \mid \bar{e}=-e\}.$$
Pick a non-zero element $\alpha \in \Q(\zeta_q)_{-}$. Now the standard
construction (see, for instance, \cite[p. 531]{Ribet3}) allows us to define the
non-degenerate $\Q(\zeta_q)$-sesquilinear Hermitian form
$$\phi_q:\H_1(J^{(f,q)},\Q) \times \H_1(J^{(f,q)},\Q) \to \Q(\zeta_q)$$
such that
$$\psi_q(x,y)=\Tr_{\Q(\zeta_q)/\Q}(\alpha \phi_q(x,y)) \ \forall x,y \in
\H_1(J^{(f,q)},\Q).$$ We write $\U(\H_1(J^{(f,q)},\Q),\phi_q)$ for the unitary
group of $\phi_q$ of the $\Q(\zeta_q)$-vector space $\H_1(J^{(f,q)},\Q)$,
viewed as an algebraic $\Q$-subgroup of $\GL(\H_1(J^{(f,q)},\Q))$ (via Weil's
restriction of scalars from $\Q(\zeta_q)^{+}$ to $\Q$ (ibid). Since the Hodge
group respects the polarization and commutes with endomorphisms of $J^{(f,q)}$,
$$\Hdg(J^{(f,q)})\subset \U(\H_1(J^{(f,q)},\Q),\phi_q).$$

Our main result is the following statement.

\begin{thm}
\label{main} Suppose that  $n\ge 4$ and   $p$ is a prime that does not divide
$n$. Let $f(x) \in \C[x]$ be a degree $n$ polynomial without multiple roots.
Let $r$ be a positive integer and $q=p^r$. Suppose that there exists a subfield
$K$ of $\C$ that contains all the coefficients of $f(x)$. Let us assume that
$f(x)$ is irreducible over $K$ and the Galois group $\Gal(f)$ of $f(x)$ over
$K$ is either $\Sn$ or  $\An$. Assume additionally that either $n \ge 5$ or
$n=4$ and $\Gal(f)=\ST_4$.

Suppose that $n>q$ and one of the following three conditions holds:
\begin{itemize}
\item[(A)] $q<n<2q$;

\item[(B)] $p$ is odd and $n\not\equiv 1 \mod q$;

\item [(C)] $p=2$, $n\not\equiv 1 \mod q$ and $n\not \equiv q-1 \mod 2q$.
\end{itemize}

Then $\Hdg(J^{(f,q)})= \U(\H_1(J^{(f,q)},\Q),\phi_q)$.
\end{thm}

\begin{rem}
The case of $q=p=3$ was earlier treated in \cite{ZarhinCrelle}.
\end{rem}

Lefschetz's theorem about algebraicity of $2$-dimensional Hodge classes and
classical invariant theory for the unitary groups \cite[Theorem 0 on p. 524;
see also pp. 531--532 ]{Ribet3} imply the following corollary to Theorem
\ref{main}.

\begin{cor}
Let $(n,p,q, f(x))$ satisfy the conditions of Theorem \ref{main}. Then every
Hodge class on each self-product $(J^{(f,q)})^m$ of $J^{(f,q)}$ can be
presented as a linear combination with rational coefficients of products of
divisor classes. In particular, the Hodge conjecture holds true for
$(J^{(f,q)})^m$.
\end{cor}

The paper is organized as follows. In Section \ref{MT} we discuss  Lie algebras
of Hodge groups of complex abelian varieties. Its main result, Theorem
\ref{hdgSS} (that may be of independent interest) asserts that under certain
conditions the semisimple part of the Hodge group (and its Lie algebra) is ``as
large as possible". We deduce Theorem \ref{main} from Theorem \ref{hdgSS},
using   auxiliary results from Section \ref {divintegral} concerning
divisibility properties of certain arithmetic functions.   In Section
\ref{linearLie} we discuss linear reductive Lie algebras. The last Section
contains the proof of Theorem \ref{hdgSS}.

\section{Complex abelian varieties}
\label{MT} Let $Z$ be a complex abelian variety of positive dimension
%%%%%%
and let $\Omega^1(Z)$ be the $\dim(Z)$-dimensional complex vector space of
regular differential $1$-forms on $Z$.
%%%%%
We write $\CC_Z$ for the center of the semisimple finite-dimensional
$\Q$-algebra $\End^0(Z)$. We write $\H_1(Z,\Q)$ for its first rational homology
group. It is well known that $\H_1(Z,\Q)$ is a $2\dim(Z)$-dimensional
$\Q$-vector space.

The $\Q$-algebra $\End^0(Z)$ acts faithfully on $\H_1(Z,\Q)$. In particular, if
 $E$ is a subfield of $\End^0(Z)$ that contains the identity map then
 $\H_1(Z,\Q)$carries the natural structure of $E$-vector space of dimension
 $$d(Z,E)=\frac{2\dim(Z)}{[E:\Q]}.$$
Let $\Sigma_E$ be the set of all field embeddings  $\sigma: E \hookrightarrow
\C$. It is well-known that
$$\C_{\sigma}:=E\otimes_{E,\sigma}\C=\C, \quad E_{\C}=E\otimes_{\Q}\C=\prod_{\sigma\in \Sigma_E}
E\otimes_{E,\sigma}\C=\prod_{\sigma\in \Sigma_E}\C_{\sigma}.$$ If $\sigma\in
\Sigma_E$ then we write $\bar{\sigma}$ for the complex-conjugate of $\sigma$.
We write  $X_E$ for the $\Q$-vector space of functions $\phi: \Sigma_E \to \Q$
with
$$\phi(\bar{\sigma})+\phi(\sigma)=0 \ \forall \sigma\in \Sigma_E.$$
If $E/\Q$ is Galois then $X_E$ carries the natural structure of $\Gal(E/\Q)$-module.

Let $\Lie(Z)$ be the tangent space to the origin of $Z$; it is a
$\dim(Z)$-dimensional $\C$-vector space. By functoriality,
$\End^0(Z)$ and therefore $E$ act on $\Lie(Z)$ and therefore
provide $\Lie(Z)$ with a natural structure of
$E\otimes_{\Q}\C$-module. Clearly,
$$\Lie(Z)=\bigoplus_{\sigma\in
\Sigma_E}\C_{\sigma}\Lie(Z)=\oplus_{\sigma\in
\Sigma_E}\Lie(Z)_{\sigma}$$ where
$\Lie(Z)_{\sigma}:=\C_{\sigma}\Lie(Z)=\{x \in \Lie(Z)\mid
ex=\sigma(e)x \quad \forall e\in E\}$. Let us put
$n_{\sigma}=n_{\sigma}(Z,E)=\dim_{\C_{\sigma}}\Lie(Z)_{\sigma}=\dim_{\C}\Lie(Z)_{\sigma}$.
  It is well-known that the natural map
$\Omega^1(Z) \to \Hom_{\C}(\Lie(Z),\C)$ is an isomorphism. This
 allows us to define via duality the natural homomorphism
$E  \to \End_{\C}(\Hom_{\C}(\Lie(Z),\C))=\End_{\C}(\Omega^1(Z))$.
This provides $\Omega^1(Z)$ with a natural structure of
$E\otimes_{\Q}\C$-module in such a way that
$\Omega^1(Z)_{\sigma}:=\C_{\sigma}\Omega^1(Z)\cong
\Hom_{\C}(\Lie(Z)_{\sigma},\C)$. In particular,
$$n_{\sigma}=\dim_{\C}(\Lie(Z)_{\sigma})=\dim_{\C}(\Omega^1(Z)_{\sigma}) \eqno (1).$$

Let us consider the first complex homology group of $Z$
$$\H_1(Z,\C)=\H_1(Z,\Q)\otimes_{\Q}\C,$$
which is a $2\dim(Z)$-dimensional complex vector space. If $E$ is as above then
$\H_1(Z,\C)$ carries the natural structure of a free
$E_{\C}:=E\otimes_{\Q}\C$-module of rank $d(Z,E)$. We have
$$\H_1(Z,\C)=\bigoplus_{\sigma\in
\Sigma_E}\C_{\sigma}\H_1(Z,\C)=\oplus_{\sigma\in \Sigma_E}\H_1(Z,\Q)_{\sigma}$$
where $$\H_1(Z,\Q)_{\sigma}:=\C_{\sigma}\H_1(Z,\C)=\{x \in \H_1(Z,\C)\mid
ex=\sigma(e)x \quad \forall e\in E\}=\H_1(Z,\Q)\otimes_{E,\sigma}\C.$$ Clearly,
every $\H_1(Z,\Q)_{\sigma}$ is a $d(Z,E)$-dimensional $\C$-vector subspace that
is also a $E_{\C}$-submodule of $\H_1(Z,\C)$.

 There is a canonical Hodge decomposition (\cite[chapter 1]{Mumford},
\cite[pp.~52--53]{Deligne})
$$\H_1(Z,\C)=H^{-1,0} \oplus H^{0,-1}$$
where $H^{-1,0}=H^{-1,0}(Z)$ and  $H^{0,-1}=H^{0,-1}(Z)$ are mutually ``complex
conjugate" $\dim(Z)$-dimensional complex vector spaces. This splitting is
$\End^0(Z)$-invariant and the $\End^0(Z)$-module $H^{-1,0}$ is canonically
isomorphic to the commutative Lie algebra $\Lie(Z)$ of $Z$. This implies that
$H^{-1,0}$ and $\Lie(Z)$ are isomorphic as $E$-modules and even as
$E_{\C}$-modules.

Let
$$\f_H^{0}=\f_{H,Z}^{0}:\H_1(Z,\C) \to
\H_1(Z,\C)$$ be the $\C$-linear operator in $\H_1(Z,\C)$ defined as follows.
$$\f_H(x) =-\frac{1}{2}x \quad \forall \ x \in H^{-1,0}; \quad \f_H^{0}(x)=\frac{1}{2}x \quad
\forall \ x \in H^{0,-1}.$$ Clearly, $\f_H^{0}$ commutes with $\End^0(Z)$. In
particular, every $\H_1(Z,\Q)_{\sigma}$ is $\f_H^{0}$-invariant. More
precisely, the linear operator $\f_H^{0}:\H_1(Z,\Q)_{\sigma}\to
\H_1(Z,\Q)_{\sigma}$ is semisimple and its spectrum lies in the two-element set
$\{1/2, -1/2\}$. Taking into account that the $E_{\C}$-modules $H^{-1,0}$ and
$\Lie(Z)$ are isomorphic, we conclude that the multiplicity of eigenvalue
$-1/2$ is $n_{\sigma}=n_{\sigma}(Z,E)$ while the multiplicity of eigenvalue
$1/2$ is $d(Z,E)-n_{\sigma}(Z,E)$. Let $\bar{\sigma}: E \hookrightarrow \C$ be
the composition of $\sigma: E \hookrightarrow \C$ and the complex conjugation
$\C \to \C$. It is known (\cite{Deligne}, \cite{MZ}) that
$$n_{\sigma}+n_{\bar{\sigma}}=d(Z,E).$$
This implies that the multiplicity of eigenvalue $1/2$ is $n_{\bar{\sigma}}$.

We refer to \cite{Ribet3}, \cite[Sect. 6.6.1 and 6.6.2]{ZarhinIzv} for the
definition and basic properties of the Hodge group (aka special Mumford--Tate
group) $\Hdg=\Hdg_Z$  of (the rational Hodge structure $\H_1(Z,\Q)$ and of)
$Z$. Recall that $\Hdg$ is a connected reductive algebraic $\Q$-subgroup of
$\GL(\H_1(Z,\Q))$, whose centralizer in $\End_{\Q}(\H_1(Z,\Q))$ coincides with
$\End^0(Z)$. Let
$$\hdg=\hdg_Z \subset \End_{\Q}(\H_1(Z,\Q))$$ be the $\Q$-Lie algebra of
$\Hdg$; it is a reductive $\Q$-Lie subalgebra of $\End_{\Q}(\H_1(Z,\Q))$, its
natural representation in  $\H_1(Z,\Q)$ is completely reducible and its
centralizer there coincides with $\End^0(Z)$. Notice also that its
complexification
$$\hdg_{\C}=\hdg\otimes_{\Q}\C\subset
\End_{\Q}(\H_1(Z,\Q))\otimes_{\Q}\C=\End_{\C}(\H_1(Z,\C))$$ contains $\f_H^0$
\cite[Sect. 3.4]{XueZ}.

Suppose that $E=\End^0(Z)$ is a CM field.  Choose a polarization on $Z$. The
corresponding Rosati involution on $\End^0(Z)$ coincides with the complex
conjugation $e\mapsto \bar{e}$ on $E$. The polarization gives rise to the
 nondegenerate alternating $\Q$-bilinear form
$$\psi: \H_1(Z,\Q) \times \H_1(Z,\Q) \to \Q$$
such that
$$\psi(ex,y)=\psi(x,\bar{e}y) \ \forall x,y \in \H_1(Z,\Q); \ e\in E.$$
  Let
$$E^{+}=\{e\in E \mid \bar{e}=e\}$$
be the maximal totally real subfield of the  CM field $E$ and let
$$E_{-}=\{e\in E \mid \bar{e}=-e\}.$$
Pick a non-zero element $\alpha \in \Q(\zeta_q)_{-}$. Now the standard
construction (see, for instance, \cite[p. 531]{Ribet3}) allows us to define the
non-degenerate $E$-sesquilinear Hermitian form
$$\phi: \H_1(Z,\Q) \times \H_1(Z,\Q) \to E$$
such that
$$\psi(x,y)=\Tr_{E/\Q}(\alpha \phi(x,y)) \ \forall x,y \in
\H_1(Z,\Q).$$ We write $\U(\H_1(Z,\Q),\phi)$ for the unitary group of $\phi$ of
the $E$-vector space $\H_1(Z,\Q)$, viewed as an algebraic $\Q$-subgroup of
$\GL(\H_1(Z,\Q))$ (via Weil's restriction of scalars from $E^{+}$ to $\Q$
(ibid). It is well-known that $\U(\H_1(Z,\Q),\phi)$ is reductive and its
$\Q$-dimension is
$$[E^{+}:\Q] d(Z,E)^2= \frac{1}{2}[E:\Q] d(Z,E)^2.$$

Let $\u(\H_1(Z,\Q),\phi)$ be the $\Q$-Lie algebra of $\U(\H_1(Z,\Q),\phi)$: it
is a reductive $\Q$-Lie subalgebra of $\End_{\Q}(\H_1(Z,\Q))$.
%, its centralizer in $\End_{\Q}(\H_1(Z,\Q))$ coincides with $E$ if $d(Z,E)>2$.
Explicitly,
$$\u(\H_1(Z,\Q),\phi)=\{u\in \End_{E}(\H_1(Z,\Q))\mid
\phi(ux,y)+\overline{\phi(x,uy)}=0 \ \forall x,y \in \H_1(Z,\Q)\}.$$ The
reductive $\Q$-Lie algebra $\u(\H_1(Z,\Q),\phi)$ splits into a direct sum
$$\u(\H_1(Z,\Q),\phi)=E_{-}\oplus \su(\H_1(Z,\Q),\phi)$$
of its center $E_{-}$ and the semisimple $\Q$-Lie algebra
$$\su(\H_1(Z,\Q),\phi)=\{u\in \u(\H_1(Z,\Q),\phi)\mid \Tr_E(u)=0.\}$$
Here
$$\Tr_E: \End_{E}(\H_1(Z,\Q))\to E$$
is the trace map.  One may easily check that
$$\dim_{\Q}(\su(\H_1(Z,\Q),\phi))=\frac{1}{2}[E:\Q] \{d(Z,E)^2-1\}.$$

 Since the Hodge group respects the polarization and commutes with
endomorphisms of $Z$,
$$\Hdg(Z)\subset \U(\H_1(Z,\Q),\phi)$$
(ibid). This implies that
$$\hdg \subset \u(\H_1(Z,\Q),\phi)\subset \End_{\Q}(\H_1(Z,\Q)).$$
This implies that the semisimple part $\hdg^{ss}=[\hdg,\hdg]$ of $\hdg$ lies in
 $\su(\H_1(Z,\Q),\phi)$. In particular,
 $$\dim_{\Q}(\hdg^{ss}) \le \frac{1}{2}[E:\Q] \{d(Z,E)^2-1\};$$
 the equality holds if and only if $\hdg^{ss}=\su(\H_1(Z,\Q),\phi)$.

The following statement may be viewed as a partial generalization of Theorem 3
in \cite[p. 526]{Ribet3}.

\begin{thm}
\label{hdgSS} Suppose that $E=\End^0(Z)$ is a CM field. Assume that all
$n_{\sigma}(Z,E)$ are distinct positive integers. Assume additionally that
there exists a field embedding $\sigma: E \hookrightarrow \C$ such that
$n_{\sigma}(Z,E)$ and $d(Z,E)$ are relatively prime.

Then $\hdg^{ss}=\su(\H_1(Z,\Q),\phi)$.
\end{thm}

\begin{rem}
\label{inequal} Clearly, in the course of the proof of Theorem \ref{hdgSS}, it
suffices to check that
$$\dim_{\Q}(\hdg^{ss}) \ge \frac{1}{2}[E:\Q] \{d(Z,E)^2-1\}.$$

\end{rem}
We prove Theorem \ref{hdgSS} in Section \ref{PhdgSS}.

\begin{proof}[Proof of Theorem \ref{main}]
Let us put $Z=J^{(f,q)}$ and $E=\Q(\zeta_q)$. Clearly,
$$d(Z,E)=n-1.$$
 We know that
$\End^0(Z)=E=\Q(\zeta_q)$ \cite{ZarhinM,ZarhinMZ2} and if $\sigma=\sigma_i:
\Q(\zeta_q)\hookrightarrow \C$ is a field embedding that sends $\zeta_q$ to
$\zeta_q^{-i}$ with $1 \le i<q, \ (i,p)=1$ then $n_{\sigma}=[ni/q]$
\cite{ZarhinM,ZarhinPisa}. (Clearly, every field embedding
$\Q(\zeta_q)\hookrightarrow \C$ is of the form $\sigma_i$.) Since $n>q$, all
integers $[ni/q]$ are positive and distinct. Propositions \ref{primeint} and
\ref{qint} below imply that under our assumptions on $(p,q,n)$
%if either $n<2q$ or $n-1$ is {\sl not} divisible by
%$p$ then
there exists a positive integer $i<q$ such that $(i,p)=1$ and $[ni/q]$
and $n-1$ are relatively prime. This allows us to apply Theorem \ref{hdgSS} and
conclude that $\hdg^{ss}=\su(\H_1(Z,\Q),\phi)$. On the other hand, by a result
from \cite{XueZ}, the center of $\hdg$ coincides with $E_{-}$. This implies
that the reductive $\Q$-Lie algebra $\hdg$ coincides with the direct sum
$$E_{-}\oplus \su(\H_1(Z,\Q),\phi)=\u(\H_1(Z,\Q),\phi).$$
Now the connectedness of the Hodge group and the unitary group implies that
$\Hdg(Z)=\U(\H_1(Z,\Q),\phi)$, i.e.,
$$\Hdg(J^{(f,q)})=\U(\H_1(J^{(f,q)},\Q),\phi_q).$$
\end{proof}

\section{Divisibility properties of integral parts}
\label{divintegral}
\begin{prop}
\label{primeint}  Suppose that $p$ is a prime, $r$ a positive integer, $q=p^r$
and $n$ is a positive integer that is not divisible by $p$.

Suppose that one of the following conditions holds:

\begin{itemize}
\item[(i)] $q<n<2q$;

\item[(ii)] The prime $p$  is odd. In addition, either $p\nmid (n-1)$
%is not divisible by $p$
or $n<2q$.
\end{itemize}

Then there exists an integer $i$ such that
$$1\le i \le q-1, \ (i,p)=1$$
and integers $[ni/p]$ and $n-1$ are relatively prime.
\end{prop}

\begin{proof}
If $q<n<2q$ then
$$[n\cdot 1/q]=[n/q]=1$$
and we may take $i=1$.

So, further we assume that $p$ is odd and either $n<q$ or $n-1$ is not
divisible by $p$.

 If $q>n>q/2$ then $[2n/q]=1$ and we may take $i=2$.

If $0<n<q/2$ then there exists a positive integer $\mu$ such that
 $$ q\le \mu n<(\mu+1)n< 2q.$$
 Since $q$ is a power of $p$ and $n$ is not divisible by $p$,
$$ q< \mu n<(\mu+1)n< 2q.$$
Clearly,
$$1=[n\mu/q]=[n(\mu+1)/q].$$
So, we take as $i$ either $\mu$ or $\mu+1$, depending on which one is {\sl not}
divisible by $p$.

Now let us assume that $n-1$ is {\sl not} divisible by $q$. Let us put
$$k=[n/q], \ c=n-kq, \ d=c-1.$$ We have
$$c=d+1, \ n=qk+c, \ n-1=qk+d; \ 2 \le c \le q-1, 1 \le d \le q-2,$$
$$(c,p)=1, \ (d,p)=1.$$
Let $i$ be an integer such that $1\le i \le q-1$ and $(i,p)=1$. Put $j=[ci/q]$.
Clearly, $j$ is a nonnegative integer such that
$$qj < ic<(q+1)j.$$
(The first inequality holds, because neither $c$ nor $i$ are  divisible
by $p$.) In other words,
$$ 0< ic-qj<q.$$
In addition,
$$[ni/q]=[(kq+c)i/q]=ik+[ci/q]=ik+j.$$
%So far, we did not use the condition that $p$ does not divide $n-1$ except the
%property
%$$(d,p)=1.$$

 Suppose that $n-1$ and $[ni/q]$ are {\sl not} relatively prime. Then there
exists a prime $\ell$ that divides both $n-1$ and $[ni/q]$. This implies that
$q\cdot k+d\cdot 1=0$ in $\Z/\ell\Z$ and $i\cdot k+j\cdot 1=0$ in $\Z/\ell\Z$.
So, we get the homogeneous system of two linear equations over the field
$\Z/\ell\Z$ that admits the {\sl nontrivial} solution $(k,1)\ne (0,0)$. By
Cramer's rule, the determinant $id-qj$ is zero in $\Z/\ell\Z$, i.e., the
integer $id-qj$ is divisible by $\ell$. In particular, $id-qj\ne \pm 1$. So, we
prove the Proposition if we find such $i$ that $id-qj$ is either $1$ or $-1$.
 In order to do that, notice that
$$id-qj=i(c-1)-qj=(ic-qj)-i.$$
 This implies that
$$-i<id-qj<q-i.$$
Since $1\le i \le q-1$,
$$1-q<id-qj<q-1.$$
 Now if we choose $i$ in such a way that  $1\le i \le q-1$ and $i d$ is
congruent to $1$ modulo $q$ (such choice is possible, because $d$ is not
divisible by $p$) then $i d-q j$ is congruent to $1$ modulo $q$ and therefore
$id-qj=1$. In addition, the latter equality implies that $i$ is {\sl not}
divisible by $p$.
\end{proof}

Proposition \ref{primeint} admits the following (partial) generalization.

\begin{prop}
\label{qint} Suppose that $p$ is a prime, $r$ a positive integer, $q=p^r$
and $n$ is a positive integer that is not divisible by $p$. Suppose that
  $n-1$ is not divisible by $q$. If $p=2$, assume additionally that
  $q=2^r>2$ and $n\not\equiv q-1 \mod 2q$.

    Then there exists an integer $i$ such that
$$1\le i \le q-1, \ (i,p)=1$$
and integers $[ni/p]$ and $n-1$ are relatively prime.
\end{prop}

\begin{rem}
\label{rempower2} If $q=p=2$ then every odd $n$ is congruent to $1$ modulo $2$.
%This implies that if $p=2$ and $(n,q)$ satisfy the conditions of Proposition
%\ref{qint} then $2^r=q>2$, i.e., $r\ge 2$.

If $p=2, q=4$ then $(n,q)$ satisfy the conditions of Proposition \ref{qint} if
and only if $n-7$ is divisible by $8$.
\end{rem}

\begin{proof}[Proof of Proposition \ref{qint}]
As in the proof of Proposition \ref{primeint}, let us put
\[ k=[n/q],\quad c=n-kq,\quad d=c-1.\]
We have
\begin{equation}
  \label{eq:2}
 (c,p)=1, \quad 2\leq c\leq q-1, \quad 1\leq d \leq q-2.
\end{equation}

We are given that $q$ does not divide $d$. However, in the light of Proposition
\ref{primeint}, we may and will assume that $p$ divides $d$; in particular,
$q>p$ and $d \ge p \ge 2>1$.

Let \[t=(d, q), \quad d^{\prime}=d/t, \quad q^{\prime}=q/t.\] Then
$$q^{\prime}>1, \ t>1, \ (d^{\prime},q^{\prime})=1$$ and both
$t$ and $q^{\prime}$ are powers of $p$. This implies that
$$(d^{\prime},p)=1, \ t \ge p\ge 2, \ q^{\prime}\ge p\ge 2 >1.$$

Since $q^{\prime}>1, \ d^{\prime}\ge 1$ and $(d^{\prime},q^{\prime})=1$, there
exists a unique pair of integers $(i,j)$ such that
\[ d^{\prime}i - q^{\prime}j=1, \quad 0<i\leq q^{\prime}-1, \quad  j\geq 0.\]
Clearly, $(i,p)=1$ and therefore $(i+q^{\prime},p)=1$, because $q^{\prime}$ is
a power of $p$. Since $t\geq 2$, we have $i+q^{\prime}< 2q^{\prime}\leq
tq^{\prime}=q$.

We will treat the case $p=2$, $n\equiv -1 \mod q$ and $n\not\equiv q-1 \mod 2q$
separately at the end. So we further assume that if $p=2$, then $n+1$ is not
divisible by $q$. We will show that either $i$ or $i+q^{\prime}$ is the desired
integer, i.e., either $[ni/q]$ and $n-1$ are relatively prime or
$[n(i+q^{\prime})/q]$ and $n-1$ are relatively prime.  It is convenient to
consider both integers $[ni/q]$ and $[n(i+q^{\prime})/q]$ as $[n(i+\epsilon
q^{\prime})/q]$ with $\epsilon=0,1$.

 For every nonnegative integer $\epsilon$ we have
\begin{equation}
  \label{eq:5}
d^{\prime}(i+ \epsilon q^{\prime})-q^{\prime}(j+ \epsilon d^{\prime})=1.
\end{equation}
%for $\epsilon=0, 1$.

Multiplying both side of (\ref{eq:5}) by $t$, we get
\begin{equation}
  \label{eq:3}
d(i+ \epsilon q^{\prime})-q(j+ \epsilon d^{\prime})=t.
\end{equation}
 Since $c=d+1$, it follows that
$$\frac{c(i+ \epsilon q^{\prime})}{q}=\frac{(d+1)(i+ \epsilon q^{\prime})}{q}=j+ \epsilon d^{\prime}+\frac{t+i+ \epsilon
q^{\prime}}{q}$$ and therefore
\begin{equation}
  \label{eq:4}
  \left[\frac{c(i+ \epsilon q^{\prime})}{q}\right]-(j+ \epsilon d^{\prime})=\left[\frac{t+i+ \epsilon q^{\prime}}{q}\right].
\end{equation}

The following Lemma will be proven at the end of this Section.

\begin{lem}
\label{t3} We keep the assumptions of Proposition \ref{qint}. If $p=2$, we
assume additionally that $q\nmid n+1$. Then $[(t+i+q^{\prime})/q]=0$ and
therefore $[(t+i)/q]=0$.
 \end{lem}

Let us assume that either $p$ is odd or $p=2$ and $n+1$ is {\sl not} divisible
by $q$.
 Combining Lemma \ref{t3} with (\ref{eq:4}), we conclude that
\begin{equation}
\label{eq:41}
 \left[\frac{c(i+ \epsilon q^{\prime})}{q}\right]=(j+ \epsilon d^{\prime})
 \end{equation}
 if $\epsilon=0$ or $1$.
It follows  that if $\epsilon=0$ or $1$ then $$[n(i+\epsilon q^{\prime})/q]=
[(kq+c)(i+\epsilon q^{\prime})/q]= k(i+\epsilon q^{\prime})+[c(i+\epsilon
q^{\prime})/q]=k(i+\epsilon q^{\prime})+(j+\epsilon d^{\prime}).$$

Now suppose that $n-1$ and $[n(i+\epsilon q^{\prime})]/q]$ are \textit{not}
relatively prime for some  $\epsilon=0$ or $1$. Then there exists a prime
$\ell$ that divides both $n-1$ and $[n(i+\epsilon q^{\prime})/q]$. This implies
that $q\cdot k+ d\cdot 1=0 $ in $\Z/\ell \Z$ and $(i+\epsilon q^{\prime})\cdot
k+(j+ \epsilon d^{\prime})\cdot 1=0 $ in $\Z/\ell \Z$. So we get the
homogeneous system of linear equations over the field $\Z/\ell\Z$ with {\sl
determinant}
$$d(i+\epsilon
q^{\prime})-q(j+ \epsilon d^{\prime})=t$$ (by (\ref{eq:3})), which admits a
non-trivial solution $(k,1)\neq (0,0)$. By Cramer's rule the {\sl determinant}
$t$ is zero in $\Z/\ell\Z$, i.e., $\ell \mid t$. Since $t$ is a power of $p$,
we conclude
 that $\ell=p$.

Since $(i+\epsilon q^{\prime})\cdot k+(j+ \epsilon d^{\prime})\cdot 1=0 $ in
$\Z/\ell \Z$ and $\ell=p$, the integer $(i+\epsilon q^{\prime})\cdot k+(j+
\epsilon d^{\prime})$ is divisible by $p$.

Now suppose that $n-1$ and $[n(i+\epsilon q^{\prime})/q]$ are \textit{not}
relatively prime for both $\epsilon=0$ and $1$. This implies that both integers
$i\cdot k+j$ and $(i+ q^{\prime})\cdot k+(j+d^{\prime})$ are divisible by $p$.
Therefore their difference $q^{\prime}k+d'$ is also divisible by $p$, which is
not the case, because $q^{\prime}$ is a power of $p$ while $(d^{\prime},p)=1$.
The obtained contradiction proves Proposition when either $p$ is odd or $p=2$
and $n+1$ is {\sl not} divisible by $q$.

At last, let us treat the remaining case when
 $p=2, q=2^r$ with $r \ge 2$, the integer  $n+1$ is divisible by $q$ but $n\not\equiv q-1\mod 2q$.
 Then $q$ divides
 $n+1$ and the ratio $k:=(n+1)/q$ is an even integer. We have $n=2^rk-1$.
Let us put  $i=2^{r-1}-1$. Since $r \ge 2$, the integer $i$ is odd. We have
\[  \left[\frac{ni}{q}\right]=\left[\frac{(2^rk-1)(2^{r-1}-1)}{2^r}\right]=(2^{r-1}-1)k-1. \]
It follows that
\[ [ni/q]\equiv -1 \mod k. \]
In particular, $[ni/q]$ is odd, since  $k$ is even. Notice that
\[(n-1)/2=[ni/q]+k. \] Combining all those assertions, we get
\[\begin{split}
 (n-1, [ni/q])&=((n-1)/2, [ni/q])=([ni/q]+k, [ni/q]) \\
&=(k,[ni/q])=(k,-1)=1.
\end{split}\]

\end{proof}

\begin{proof}[Proof of Lemma \ref{t3}]
Recall that $q^{\prime}$ and $t$ are powers of $p$. This implies that $t \ge
p\ge 3$ if $p$ is odd; if $p=2$ then either $t=2$ or $t \ge 4$.

First, let us assume that $p$ is odd and therefore
 $t\geq p \ge 3$ and $q^{\prime} \geq p\ge 3$.  It follows that
\[ q^{\prime}(t-2) \geq 3(t-2)= 3t-6 \geq t. \]
Since $i\leq q^{\prime}-1$,
\[q=tq^{\prime} \geq t+2q^{\prime}>t+i+q^{\prime}.\]
This implies that $[(t+i+q^{\prime})/q]=0$.

Second, assume that $p=2$ and $t\ge 4$. Then $q^{\prime}\ge 2$ and therefore
$$q^{\prime}(t-2) \geq 2(t-2)=2t-4 \ge t.$$
As above, $i\leq q^{\prime}-1$,
\[q=tq^{\prime}=(t-2)q^{\prime}+2q^{\prime} \geq t+2q^{\prime}>t+i+q^{\prime}\]
and therefore $[(t+i+q^{\prime})/q]=0$.

Third, assume that $p=2$ and  $t=2$. Then
$$t=2, \ d=2d^{\prime}, \ q^{\prime}=\frac{q}{2}=2^{r-1}.$$  Recall
that $(i, p)=1$, i. e.,  $i$ is odd. Since $i\le q^{\prime}-1$, the sum
\[t+i+q^{\prime}=2+i+2^{r-1}\] is greater or equal than $q=2^r$ only if $i=q^{\prime}-1=2^{r-1}-1$.
By (\ref{eq:5}), $d^{\prime}\cdot i\equiv 1 \mod q^{\prime}$. If
$i=q^{\prime}-1$, then $d^{\prime}=q^{\prime}-1$, since $1\leq d^{\prime}
<q^{\prime}$. This implies that
\[c=d+1=2d^{\prime}+1=2q^{\prime}-1=q-1,\] which contradicts the assumption that
$$n+1=(kq+c)+1=kq+(c+1)$$ is {\sl not} divisible by $q$. So, this case does
{\sl not} occur.

\end{proof}

%%%%%%%%%%%%%%%%%%%%%%%%%%%%%%%%%%%%%%%%%%%%%
\section{Linear reductive Lie algebras}
\label{linearLie} Throughout this Section, $Q$ is a field of characteristic
zero, $C$ is an algebraically closed field containing $Q$. If $W$ is a
$Q$-vector space (resp. $Q$-algebra or $Q$-Lie algebra) then we write $W_C$ for
$W\otimes_Q C$ provided with the natural structure of a $C$-vector space (resp.
$C$-algebra or $C$-Lie algebra).

Let $W$ be a nonzero finite-dimensional $Q$-vector space. Let $E \subset
\End_Q(W)$ be a subfield that contains the scalars $Q\cdot\I_W$. Then $E/Q$ is
a finite algebraic extension and $W$ carries the natural structure of
$E$-vector space; in addition,
$$\dim_Q(W)=[E:Q] \cdot \dim_E(W).$$

We write $\Sigma$ for the set of all $Q$-linear field embedding
$\sigma:E\hookrightarrow  C$. If $\sigma \in \Sigma$ then
 we write $W_{\sigma}$ for the $C$-vector
space $W\otimes_{E,\sigma}C$; clearly,
$$\dim_{C}(W_{\sigma})= \dim_E(W);$$
there are natural surjective homomorphisms $W_C \twoheadrightarrow W_{\sigma}$.
 Their ``direct sum"
$$W_C \to \oplus_{\sigma\in\Sigma} W_{\sigma}$$
is an isomorphism of $C$-vector spaces, so one may view every $W_{\sigma}$ as a
$C$-vector subspace (direct summand) of $W_C$.

\begin{rem}
\label{indep}
 Let $S$ be a $Q$-vector subspace of $W$. Let $E S\subset W$ be the $E$-vector subspace
 of $W$ generated by $S$, i.e., $E S$ is the set of all linear combinations of
 elements of $S$ with coefficients in $E$.
 Clearly, $E S=W$ if and only if $S$ contains a basis of the $E$-vector space
 $W$. It is also clear that the image of
 the composition $$S_C \subset W_C \twoheadrightarrow W_{\sigma}$$
 coincides
with
$$\{ES\}_{\sigma}=ES \otimes_{E,\sigma}C\subset W\otimes_{E,\sigma}C=
W_{\sigma};$$ in particular, the $C$-dimension of this image coincides with
$\dim_E (ES)$ and does not depend on the choice of $\sigma$.
\end{rem}

\begin{rem}
\label{irred} Let $\k \subset \End_Q(W)$ be a reductive $Q$-Lie subalgebra such
that the natural representation of $\k$ is completely reducible and the
centralizer $\End_{\k}(W)$ of $\k$ in $\End_Q(W)$ coincides with $E$; in
particular, the $\k$-module $W$ is simple and $\k \subset \End_E(W)$. Clearly,
$E\k$ is a reductive $E$-Lie subalgebra of $\End_E(W)$ and the centralizer of
$E\k$ in $\End_E(W)$ coincides with $E$. In other words, the $E$-vector space
$W$ is an absolutely simple $E\k$-module.This implies that the $C$-vector space
$W_{\sigma}$ is an absolutely simple $\{E\k\}_{\sigma}$-module. However,
applying Remark \ref{indep} to the $E$-vector space $\End_E(W)$ (instead of
$W$) and its $Q$-vector subspace $S=\k$, we conclude that the $C$-Lie
subalgebra $\{E\k\}_{\sigma}\subset \End_{C}(W_{\sigma})$ is the image of
$$\k_C \to \End_E(W) \otimes_{E,\sigma}C=\End_C(W_{\sigma}).$$ This implies
that the $\k_C$-module $W_{\sigma}$ is also absolutely simple.  On the other
hand, the reductiveness of $\k$ and the equality $\End_{\k}(W)=E$ imply that
$\k$ splits into the direct sum
$$\k=\k^{ss}\oplus \c$$
of its center $\c$ and the semisimple $Q$-Lie algebra $\k^{ss}=[\k,\k]$; in
addition, $\c\subset E$. This implies that $E\k=E\k^{ss}\oplus E\c$ where $E\c$
is either $E$ or $\{0\}$. It follows easily that in both cases the
$\k^{ss}_C$-module $W_{\sigma}$ remains  absolutely simple and the image of
$\k^{ss}_C$ in $\End_C(W_{\sigma})$ coincides with
$$\{E\k^{ss}\}_{\sigma}=E\k^{ss}\otimes_{E,\sigma}C=[\{E\k\}_{\sigma},\{E\k\}_{\sigma}].$$

\end{rem}

The following statement is a variant of \cite[Sect. 4, Prop. 5]{SerreLF}

\begin{lem}
\label{serre} Let $W$ be a finite-dimensional $C$-vector space of positive
dimension, let $\sL(W)$ be the Lie algebra of traceless linear operators in
$W$. Let $\g \subset \End_C(W)$ be an irreducible semisimple linear $C$-Lie
subalgebra. Assume that there exists a diagonalizable operator
$$f \in \g \subset \End_C(W)$$
that enjoys the following property: $f$ acts in $W$ as a linear operator with
exactly two eigenvalues, whose multiplicities are relatively prime.

Then $\g=\sL(W)$.
\end{lem}

\begin{proof}[Proof of Lemma \ref{serre}]

Clearly, $f \ne 0$ (otherwise, it would have only one eigenvalue).

Let $\alpha_1$ and $\alpha_2$ be the eigenvalues of $f$ and let $W_1$ and $W_2$
are the corresponding eigenspaces. Since $f \ne 0$, either $\alpha_1\ne 0$ or
$\alpha_2 \ne 0$.
%%%%%%%%%%%%%%%%%revised January 2010%%%%%%%%%%%%%%%%%%%%%%%%%%%%%%
We claim that $\g$ is simple. Indeed, let us split the semisimple $\g$ into a
direct sum $\g =\oplus_{i=1}^r \g_i$ of simple $C$-Lie algebras $\g_i$. Then
the simple $\g$-module $W$ splits into a tensor product $W=\otimes _{i=1}^r
U_i$ of simple $\g_i$-modules $U_i$. Since $W$ is a faithful $\g$-module, all
$\dim_C(U_i)>1$.  We have $f=\sum_{i=1}^r f_i$ with $f_i\in \g_i$. Clearly, if
some $f_i=0$ then the multiplicities of all eigenvalues of $f$ are divisible by
$\dim_C(U_i)$. However, the spectrum of $f$ consists of  (only) two
eigenvalues, whose multiplicities are relatively prime. This implies that all
$f_i \ne 0$, i.e., $f$ does {\sl not} belong to a proper ideal of $\g$.
 Now the simplicity of $\g$ follows from \cite[Th. 1.5 on p. 286]{ZarhinChicago} (with $k=C$ and $V=W$).
% Since the spectrum of $f$ consists of  two eigenvalues, it follows that $\g$
%is a {\sl simple} Lie algebra \cite[Th. 1.5 on p. 286]{ZarhinChicago}.

%%%%%%%%%%%%%%%%%%%%%%%%%%%%%%%%%%%%%%%%%%

The semisimplicity of $f$ means that
$$W=W_1\oplus W_2.$$
Let us put
$$n= \dim_C(W_1), \ m= \dim_C(W_2).$$
By assumption, $n$ and $m$ are positive integers that are relatively prime. Let
us put
$$h:=n+m =\dim_C(W).$$
Clearly, $h$ and $m$ are relatively prime.

 The semisimplicity of $\g$ implies that $\g \subset \sL(W)$; in particular,
the trace of $f$ is zero. This means that
$$n \alpha_1 + m \alpha_2 =0.$$
It follows that both $\alpha_1$ and $\alpha_2$ do not vanish. Replacing $f$ by
$(m/\alpha_1)f$, we may and will assume that the spectrum of $f$ consists of
eigenvalue $m$ of multiplicity $n$ and eigenvalue $-n$ of multiplicity $m$. In
addition, $W_1$ is the eigenspace attached to eigenvalue $m$ and $W_2$ is the
eigenspace attached to eigenvalue $-n$.

Let $\Theta \subset \GL(W)$ be the linear algebraic subgroup that consists of
all linear operators $A(u,v)$ that act on $W_1$ as multiplication by $u$ and on
$W_2$ as multiplication by $v$ where $(u,v) \in C^{*} \times C^{*}$ satisfy
$$u^{n}v^m=1.$$
Let us consider the multiplicative algebraic group $\G_m$ over $C$.
 The morphism of algebraic $C$-groups
$$\rho:\G_m \to \Theta, \ z \mapsto A(z^m, z^{-n})$$
is an isomorphism: in order to construct the inverse map, pick integers $a$ and
$b$ with $ma-nb=1$ (here we use the assumption that $n$ and $m$ are relatively
prime). Then the morphism of algebraic group
$$\Theta \to \G_m,  \ A(u,v) \mapsto u^a  v^b$$
is the inverse of $\rho$. So, $\Theta$ is isomorphic to $\G_m$; in particular,
it is a one-dimensional connected linear algebraic group and its Lie algebra is
a one-dimensional $C$-Lie subalgebra of $\End_C(W)$. Considering the tangent
map to the composition
$$\G_m \to \Theta \subset \GL(W),$$
one may easily find that the Lie algebra $\Lie(\Theta)$ of $\Theta$ coincides
with $C \cdot f\subset \End_C(W)$. On the other hand, let $\mu_h \subset C^{*}$
be the order $h$ cyclic group of $h$th root of unity. If $z \in \mu_h$   then
$z^m= z^{-n}$ (since $n+m=h$) and therefore  $A(z^m, z^{-n})$ coincides with
multiplication by $z^m$. This implies that
$$\mu_h \cdot\I \subset \Theta \subset \GL(W)$$
where $\I$ is the identity map in $W$.

Since every (linear) simple Lie algebra is algebraic, there exists a connected
simple linear algebraic subgroup $\GG \subset \GL(W)$, whose Lie algebra
coincides with $\g$. Since $f \in \g$, we have $C\cdot f \subset \g$. In other
words, the Lie algebra of connected $\Theta$ lies in the Lie algebra of $\GG$.
This implies that $\Theta \subset \GG$ and therefore
$$\mu_h \cdot\I \subset \Theta \subset \GG\subset \GL(W);$$
in particular, the order of the center of simple $\GG$ is divisible by
$h=\dim_C(W)$. It follows from Lemma 2 of Sect. 4 in \cite{SerreLF} that
$\GG=\SL(W)$. Since the Lie algebra of $\SL(W)$ coincides with $\sL(W)$, we
conclude that $\g=\sL(W)$.

\end{proof}

\begin{thm}
\label{slT}
 Let $V$ be a  $Q$-vector space of positive finite dimension and let
$\k\subset\End_Q(V)$ a reductive Lie algebra, whose natural representation in
$V$ is completely reducible. Suppose that the centralizer
$$E:=\End_{\k}(V)$$
is a field.
%that is normal over $Q$.
Let us put
$$h=\dim_E(V).$$
Let
$$f \in \k_C\subset\End_Q(V)\otimes_Q C=\End_C(V_C)$$
be a non-zero semisimple element that enjoys the following properties.

\begin{itemize}
\item[(i)] The spectrum of the $C$-linear operator $f:V_C \to V_C$ consists of
two eigenvalues, $\lambda$ and $\mu$;
 \item[(ii)] For every $\sigma \in \Sigma$
%embedding $\sigma:E\hookrightarrow C$
we write $n_{\sigma}$ and $m_{\sigma}$ for the multiplicities of of eigenvalues
$\lambda$ and $\mu$ of the $C$-linear operator $f: V_{\sigma} \to V_{\sigma}$ .
Then:
\begin{enumerate}
\item all  the numbers $\{n_{\sigma}\}_{\sigma\in\Sigma}$ are distinct positive
integers. Assume also that all $m_{\sigma}$ are positive integers.

\item There exists $\tau\in\Sigma$ such that $n_{\tau}$ and $m_{\tau}$ are
relatively prime.
\end{enumerate}
\end{itemize}

Then the semisimple part $\k^{ss}_C=[\k_C,\k_C]$ of $\k_C$ contains an ideal
that is isomorphic to a direct sum of $r$ copies of the $C$-Lie algebra
$\sL(h,C)$ of traceless matrices of size $h$ with $r \ge [E:Q]/2$. In
particular, if $\k^{ss}=[\k,\k]$ is the semisimple part of $\k$ then
$$\dim_Q(\k^{ss})=\dim_C(\k^{ss}_C)\ge \frac{1}{2}[E:Q] (h^2-1).$$

\end{thm}

In order to prove Theorem \ref{slT}, we need the following two statements.

\begin{lem}
\label{graph}
 Let $\g_1$ and $\g_2$ be non-zero finite-dimensional simple Lie
algebras over a field of characteristic zero. Let $\g \subset \g_1\oplus \g_2$
be a semisimple Lie subalgebra such that the both projection maps
$$\g \to \g_1, \ \g \to \g_2$$
are surjective. Then either $\g = \g_1\oplus \g_2$ or there exists a Lie
algebra isomorphism $\phi: \g_1 \cong \g_2$ such that $\g$ coincides with the
graph $\Gamma(\phi)$ of $\phi$ in $\g_1\times  \g_2=\g_1\oplus \g_2$.

In particular, if $\g_1$ and $\g_2$ are not isomorphic then $\g = \g_1\oplus
\g_2$.
\end{lem}

\begin{lem}
\label{slL} Let $n\ge 2$ and $d$ be positive integers. Let $a_1, \dots , a_d$
be $d$ distinct positive integers such that
$$1 \le a_i < n \ \forall i; \quad a_i \ne n-a_j \ \forall i,j.$$
Let $C$ be an algebraically closed field of characteristic zero and $W_1, \dots
, W_d$ be $n$-dimensional $C$-vector spaces. Let us put
$$W =\oplus_{i=1}^d W_i$$
and let
$$\k^{ss} \subset \oplus_{i=1}^d \sL(W_i) \subset \oplus_{i=1}^d \End_C(W_i)\subset
\End_C(W)$$ be a semisimple $C$-Lie algebra that enjoys the following
properties:
\begin{itemize}
\item[(i)] The projection map $\k^{ss} \to \sL(W_i)$ is surjective for all $i$.
\item[(ii)] There exists a semisimple element
$$f \in \k^{ss} \subset \oplus_{i=1}^d \End_C(W_i)$$
such that for all $i$ the element $f$ acts in $W_i$ as a linear operator with
two eigenvalues of multiplicity $a_i$ and $n-a_i$ respectively.
\end{itemize}
Then $\k^{ss} = \oplus_{i=1}^d \sL(W_i)$.
\end{lem}

\begin{proof}[Proof of Lemma \ref{graph}]
Let $\g_0$ be the kernel of the first projection map $\g \twoheadrightarrow
\g_1$. By definition,
$$\g_0 \subset \{0\}\oplus \g_2 \subset \g_1\oplus
\g_2.$$ The surjectivity of the second projection $\g \to \g_2$ implies that
$\g_0$ is an ideal in $\{0\}\oplus\g_2 \cong \g_2$. The simplicity of $\g_2$
implies that either $\g_0=\{0\}$ or $\g_0=\{0\}\oplus\g_2$. In the latter case
$\g=\g_1\oplus \g_2$. So, let us assume that $\g_0=\{0\}$, i.e., the first
projection map is an isomorphism. This means that there is a Lie algebra
homomorphism $\phi: \g_1\to \g_2$ such that $\g$ coincides with the graph
$\Gamma(\phi)$ of $\phi$. Now the surjectiveness of the second projection map
means that $\phi$ is surjective. Since $\g_1$ is simple, $\phi$ is injective
and therefore is an isomorphism.
\end{proof}

\begin{proof}[Proof of Lemma \ref{slL}]
Let us denote by $f_i:W_i \to W_i$ the linear operator in $W_i$ induced by $f$
(for all $i$).

  If $d=1$ then the result follows from the
property (i). Assume now that $d=2$. If $\k^{ss}\ne \sL(W_1)\oplus \sL(W_2)$
then it follows from Lemma \ref{graph}  that $\k^{ss}$ coincides with the graph
$\Gamma(\phi)$ of a certain Lie algebra isomorphism $\phi:\sL(W_1)\cong
\sL(W_2)$. It is well known that such a
 Lie algebra isomorphism is induced either by an isomorphism of vector spaces
 either between $W_1$ or $W_2$ or between $W_1^{*}=\Hom_C(W_1,C)$ and $W_2$. In
 the former case the spectra of $f_1$ and $f_2$ coincide (including the
 multiplicities). In the latter case the spectra of $-f_1$ and $f_2$ coincide (including the
 multiplicities). This implies that either $a_1=a_2$ or $a_1=n-a_2$. This
 contradicts our assumptions and proves the case of $d=2$. In the case of
 arbitrary $d \ge 2$, let us apply Lemma \ref{graph} to the image of $\k^{ss}$
 in $\sL(W_i)\oplus \sL(W_j)$: we obtain
   that for every pair of distinct indices $i,j \le d$
  the projection map $\k^{ss} \to \sL(W_i)\oplus \sL(W_j)$ is surjective. Now
 the case of arbitrary $d$ follows from Lemma on pp. 790--791 of \cite{Ribet}.
\end{proof}

\begin{proof}[Proof of Theorem \ref{slT}]
Let us split the reductive $Q$-Lie algebra $\k_C$ Lie into a direct sum
$$\k_C=\k^{ss}\oplus \c$$
of its center $\c\subset E$ and the semisimple $Q$-Lie algebra
$$\k^{ss}=[\k,\k]=[\k,\k]$$ Then
$f=f_0+f_{\c}$ with $f_0\in \k^{ss}_C, f_{\c}\in \c_C$. By Remark \ref{irred},
 the $\k_C$-module $V_{\sigma}$ is absolutely simple for all $\sigma\in
 \Sigma$.
By Schur's Lemma  there exists $c_{\sigma} \in C$ such that $f_{\c}$ acts in
$V_{\sigma}$ as multiplication by $c_{\sigma}$. It follows that $f_0$ acts in
$V_{\sigma}$ as a diagonalizable operator with eigenvalues $\lambda-c_{\sigma}$
of positive multiplicity $n_{\sigma}$ and $\mu-c_{\sigma}$ of positive
multiplicity $m_{\sigma}$. Let us denote by $f_{\sigma}$ the linear operator in
$V_{\sigma}$ induced by $f_0$. It follows from Remark \ref{irred} that
$$f_{\sigma} \in \{E\k^{ss}\}_{\sigma} \subset \End_{C}(V_{\sigma})$$
and $\{E\k^{ss}\}_{\sigma}$ is an irreducible linear semisimple Lie
$C$-(sub)algebra. It is also clear that $f_{\sigma}$ is a diagonalizable
operator with (exactly two) eigenvalues $\lambda-c_{\sigma}$ of multiplicity
$n_{\sigma}$ and $\mu-c_{\sigma}$ of multiplicity $m_{\sigma}$.

Taking $\sigma=\tau$ and applying Lemma \ref{serre} to $W=W_{\tau}$,
$f=f_{\tau}$ and $\g=\{E\k^{ss}\}_{\sigma}$, we conclude that
$\{E\k^{ss}\}_{\tau}=\sL(W_{\tau})$. In other words, the image of $\k^{ss}_C
\to\End_{C}(V_{\tau})$ coincides with $\sL(V_{\tau})$; in particular, the
$C$-dimension of the image is $h^2-1$. By Remark \ref{indep}, for all $\sigma
\in C$ the image of $\k^{ss}_C \to\End_{C}(V_{\sigma})$ also has $C$-dimension
$h^2-1$. The semisimplicity of $\k^{ss}_C$ implies the semisimplity of the
image and therefore this image must lie in $\sL(V_{\sigma})$; in particular,
its $C$-dimension does not exceed $h^2-1$ (recall that $h=\dim_C(V_{\sigma})$
). It follows that the image of $\k^{ss}_C \to\End_{C}(V_{\sigma})$ coincides
with $\sL(V_{\sigma})$ for all $\sigma\in \Sigma$.

Now let us choose  a maximal subset $\Pi\subset \Sigma$ with respect to the
following property:
$$n_{\sigma} \ne m_{\kappa} \ \forall \sigma, \kappa \in \Pi.$$
%%%%%%%%%%%%%%%%%%%%%%%%
We claim that $\#(\Pi) \ge \#(\Sigma)/2$. Indeed, if $\#(\Pi) < \#(\Sigma)/2$
then the cardinality of the set
$$N_{\Pi} =\{n_{\sigma}\mid \sigma \in \Pi\} \cup \{m_{\sigma} \mid \sigma \in \Pi\}$$
does not exceed
 $$2 \cdot \#(\Pi)< 2 \cdot \frac{\#(\Sigma)}{2}=\#(\Sigma).$$
  This implies that
 $\#(N_{\Pi})<\#(\Sigma)$ and therefore
 there
exists $\kappa \in \Sigma$ such that $n_{\kappa}$ does {\sl not} belong to
$N_{\Pi}$; in particular, $\kappa$ does not belong to $\Pi$. It follows from
the very definition of $N_{\Pi}$ that $m_{\kappa}$ also does {\sl not} belong
to $N_{\Pi}$. This implies that $\Pi$ is not maximal, because we could replace
it by $\Pi \cup \{\kappa\}$. The obtained contradiction proves the desired
inequality.

Since $\#(\Sigma)=[E:Q]$, we have $\#(\Pi) \ge [E:Q]/2$. Now let us put
$$W:=\oplus_{\sigma\in\Pi}V_{\sigma}$$
and denote by $\k_{\Pi}$ the image of $\k^{ss}_C$ in
$$\oplus_{\sigma\in\Pi}\End_{C}(V_{\sigma})\subset
\End_{C}(\oplus_{\sigma\in\Pi}V_{\sigma})=\End_C(W).$$ Clearly, the linear
operator $f_{\Pi}:W \to W$ induced by $f$ is a semisimple element of
$\k_{\Pi}$.
 Now the result follows from
Lemma
 \ref{slL} applied to $W$,  the
 semisimple $C$-Lie algebra $\k_{\Pi}\subset \End_{C}(W)$ and  $f_{\Pi}$.
\end{proof}

%%%%%%%%%%%%%%%%%%%%%%%%%%%%%%%%%%%%%%%%%%%%%
The following statement will be used in Section \ref{PhdgSS}.

\begin{lem}
\label{simplePROD} Let $r \ge 2$ be a positive integer and let $\g_1, \dots,
g_r$ be mutually nonisomorphic finite-dimensional simple $Q$-Lie algebras. Let
$\g \subset \oplus_{i=1}^r\g_i$ be a semisimple $Q$-Lie subalgebra such that
every projection map $\g \to \g_i$ is surjective. Then $\g= \oplus_{i=1}^r\g_i$
\end{lem}

\begin{proof}
Let $i,j \le r$ be two distinct positive integers. Applying Lemma \ref{graph}
to the image of $\g$ in $\g_i\oplus\g_j$ (with respect to the corresponding
projection map), we conclude that the projection map $\g \to g_i\oplus\g_j$ is
surjective. Now the result follows from Lemma on pp. 790--791 of \cite{Ribet}.

\end{proof}

\section{Semisimple components of Hodge groups and their Lie algebras}
\label{PhdgSS}

\begin{proof}[Proof of Theorem \ref{hdgSS}] Let us apply Theorem \ref{slT}  to
$$Q=\Q, \ C=\C, \ V=\H_1(Z,\Q), \ h=d(Z,E),$$
$$ \k=\hdg,\  f =\f_H^0, \ \lambda=-\frac{1}{2}, \ \mu=\frac{1}{2},$$
$$n_{\sigma}=n_{\sigma}(Z,E), \
m_{\sigma}=d(Z,E)-n_{\sigma}(Z,E)=n_{\bar{\sigma}}(Z,E).$$ We conclude that
$$\dim_{\Q}(\hdg^{ss}) \ge \frac{1}{2}[E:\Q] \{d(Z,E)^2-1\}.$$
By Remark \ref{inequal}, this implies that
$$\hdg^{ss}=\su(\H_1(Z,\Q), \phi)$$ and we are done.
\end{proof}

\begin{sect}
\label{prodHDG} Let us assume that  $p$ is odd and $q=p^r$
%>p$, i.e., $r>1$
and consider the abelian variety $Z=\prod_{i=1}^r J^{(f,p^i)}$ and its first
rational homology group $\H_1(Z,\Q)=\oplus_{i=1}^r \H_1(J^{(f,p^i)},\Q)$. Then
every subspace $\H_1(J^{(f,p^i)},\Q)$ is $\Hdg(Z)$-invariant and the image of
$\Hdg(Z)$ in $\GL_{\Q}(\H_1(J^{(f,p^i)},\Q))$ coincides with
$\Hdg(J^{(f,p^i)})$. (This assertion follows easily from the minimality
property of Hodge groups.) It follows that every  $\H_1(J^{(f,p^i)},\Q)$ is
$\hdg_Z$-invariant and the image of $\hdg_Z$ in
$\End_{\Q}(\H_1(J^{(f,p^i)},\Q))$ coincides with the Lie algebra of
$\Hdg(J^{(f,p^i)})$. This implies that the image of the semisimple $\Q$-Lie
algebra $\hdg_Z^{ss}$ in $\End_{\Q}(\H_1(J^{(f,p^i)},\Q))$ coincides with  the
semisimple part of the Lie algebra of $\Hdg(J^{(f,p^i)})$.
\end{sect}

\begin{thm}
\label{mainP}
 Suppose that $p$ is an odd prime, $n \ge 4$ is a positive integer
such that $p$ does not divide $n(n-1)$. Suppose that $f(x) \in \C[x]$ is a
degree $n$ polynomial without multiple roots.  Suppose that there exists a
subfield $K$ of $\C$ that contains all the coefficients of $f(x)$. Let us
assume that $f(x)$ is irreducible over $K$ and the Galois group $\Gal(f)$ of
$f(x)$ over $K$ is either $\Sn$ or  $\An$. Assume additionally that either $n
\ge 5$ or $n=4$ and $\Gal(f)=\ST_4$. Let $r$ be a positive integer and $q=p^r$.
Let us put $Z=\prod_{i=1}^r J^{(f,p^i)}$. If $n>q$ then
$${\hdg_Z}^{ss} = \oplus_{i=1}^r \su(\H_1(J^{(f,p^i)},\Q),\phi_{p^i}).$$
\end{thm}

\begin{proof}
It follows from Theorem \ref{main} that the semisimple part of the Lie algebra
of $\Hdg(J^{(f,p^i)})$ coincides with $\su(\H_1(J^{(f,p^i)},\Q),\phi_{p^i})$
for all $i \le r$; notice that all $\su(\H_1(J^{(f,p^i)},\Q),\phi_{p^i})$'s are
mutually nonisomorphic simple $\Q$-Lie algebras. Now the result follows from
arguments of Subsect. \ref{prodHDG} combined with Lemma \ref{simplePROD}.
\end{proof}

We keep the notation and assumptions of Theorem \ref{mainP}. For every positive
integer $i$ let us  put
\[
\Q(\zeta_{p^i})_{-} := \{e \in \Q(\zeta_{p^i})\mid
\bar{e}=-e\}\subset \Q(\zeta_{p^i})=\End^0(J^{(f,p^i)})\subset
\End_{\Q}(\H_1(J^{(f,p^i)},\Q)).\] Let us put
\begin{gather*}
\Tr_i=\Tr_{\Q(\zeta_{p^{i+1}})/\Q(\zeta_{p^i})}:
\Q(\zeta_{p^{i+1}})\to \Q(\zeta_{p^i}),\\
\E^{p,r}_{-} :=\{ (e_i)_{i=1}^r  \in \oplus_{i=1}^r \Q(\zeta_{p^i})_{-}\mid
\Tr_i(e_{i+1})=e_i \ \forall i<r\}\subset
\\
\oplus_{i=1}^r \Q(\zeta_{p^i})_{-} \subset \oplus_{i=1}^r
\End_{\Q}(\H_1(J^{(f,p^i)},\Q))\subset \End_{\Q}(\H_1(Z,\Q)).
\end{gather*}
Recall \cite{XueZ} that (under our assumptions) the center of the
$\Q$-Lie algebra $\hdg_Z$ coincides with
$$\E^{p,r}_{-}\subset
\End_{\Q}(\H_1(Z,\Q)).$$ Then the reductiveness of $\hdg_Z$ combined with
Theorem \ref{mainP} implies the following statement.

\begin{thm}
\label{mainPP} Suppose that $p$ is an odd prime, $n \ge 4$ is a positive
integer such that $p$ does not divide $n(n-1)$. Suppose that $f(x) \in \C[x]$
is a degree $n$ polynomial without multiple roots.  Suppose that there exists a
subfield $K$ of $\C$ that contains all the coefficients of $f(x)$. Let us
assume that $f(x)$ is irreducible over $K$ and the Galois group $\Gal(f)$ of
$f(x)$ over $K$ is either $\Sn$ or  $\An$. Assume additionally that either $n
\ge 5$ or $n=4$ and $\Gal(f)=\ST_4$. Let $r$ be a positive integer and $q=p^r$.
Let us put $Z=\prod_{i=1}^r J^{(f,p^i)}$. If $n>q$ then
$${\hdg_Z} =\E^{p,r}_{-}\oplus [\oplus_{i=1}^r \su(\H_1(J^{(f,p^i)},\Q),\phi_{p^i})].$$
\end{thm}

\begin{rem} \label{isog} We keep the assumptions of Theorem \ref{mainPP}. Let us
fix an isogeny $\alpha: J(C_{f,p^r})\to \prod_{i=1}^r J^{(f,p^i)}=Z$. Then
$\alpha$ induces an isomorphism of $\Q$-vector spaces $\alpha:
\H_1(J(C_{f,p^r}),\Q) \cong \H_1(Z,\Q)$. Clearly, the Hodge group of
$\Hdg(J(C_{f,p^r}))$ coincides with $\alpha^{-1}\Hdg(Z)\alpha$. This implies
that the $\Q$-Lie algebra of $\Hdg(J(C_{f,p^r})))$ coincides with $$\alpha
\{\E^{p,r}_{-}\oplus [\oplus_{i=1}^r
\su(\H_1(J^{(f,p^i)},\Q),\phi_{p^i})]\}\alpha^{-1}.$$
\end{rem}


\begin{thebibliography}{99}

\bibitem{Deligne} P.~Deligne, {\em Hodge cycles on abelian varieties} (notes by
J.\/S.~Milne). Lecture Notes in Math., vol. {\bf 900} (Springer-Verlag, 1982),
pp. 9--100.



\bibitem{MZ} B. Moonen, Yu. G. Zarhin, {\em Weil classes on abelian varieties}.  J. reine angew.  Math.
{\bf 496} (1998), 83-92.



\bibitem{Mumford} Mumford, D.: Abelian varieties, Second edition.
 Oxford University Press, London, 1974.


 \bibitem{Ribet} K. Ribet,  {\em Galois action on division points of Abelian varieties with real
multiplications}. Amer. J. Math. {\bf 98}, 751--804 (1976).

\bibitem{Ribet3} K. Ribet, {\em Hodge classes on certain abelian varieties}.  Amer. J. Math. {\bf 105} (1983), 523--538.




\bibitem{SerreLF} Serre, J.-P. {\em Sur les groupes de Galois attach\'es aux groupes $p$-divisibles}.
Proc. Conf. Local Fields (Driebergen, 1966),  pp. 118--131, Springer,
 Berlin, 1967.

\bibitem{Serre} Serre, J.-P.: Repr\'esentations lin\'eares des groupes
finis, Troisi\'eme \'edition. Hermann, Paris, 1978.



\bibitem{Shimura} G. Shimura, Abelian varieties with complex
multiplication and modular functions. Princeton University Press,
Princeton, 1997.



\bibitem{XueZ} J. Xue, Yu.G. Zarhin, {\em Centers of Hodge groups of superelliptic
jacobians}. Transformation Groups, to appear;    arXiv:0907.1563 [math.AG].

\bibitem{ZarhinIzv} Yu.G. Zarhin, {\em Weights of simple Lie algebras in the cohomology of
algebraic varieties}. Izv. Akad. Nauk SSSR Ser. Mat. {\bf 48} (1984), 264--304;
 Math. USSR Izv. {\bf 24} (1985), 245 - 281.



\bibitem{ZarhinChicago} Yu.G. Zarhin, {\em Linear irreducible Lie algebras and Hodge structures}.
Algebraic geometry (Chicago, IL, 1989),  281--297, Lecture Notes in Math. {\bf
1479}, Springer, Berlin, 1991.

\bibitem{ZarhinMRL} Yu.G. Zarhin, {\em Hyperelliptic jacobians without complex multiplication}.
 Math. Res. Letters {\bf 7} (2000), 123--132.

\bibitem{ZarhinMMJ} Yu.G. Zarhin, {\em Very simple $2$-adic representations
 and hyperelliptic jacobians}. Moscow Math. J. {\bf 2} (2002), issue 2, 403-431.

\bibitem{ZarhinCrelle} Yu.G. Zarhin, {\em Cyclic covers, their Jacobians and endomorphisms}.  J.
reine angew.  Math. {\bf 544} (2002), 91--110 .

\bibitem{ZarhinCamb} Yu.G. Zarhin,  {\em The endomorphism rings of Jacobians of cyclic covers of the
projective line}. Math. Proc. Cambridge Philos. Soc. {\bf 136} (2004),
257--267.



%\bibitem{ZarhinSh}  Yu.G. Zarhin, {\em Homomorphisms of
%hyperelliptic Jacobians}.  In: Number Theory, Algebra and Algebraic Geometry
%(Shafarevich Festschrift). Trudy Mat. Inst. Steklov {\bf 241} (2003), 90--104;
%Proc. Steklov Inst. Math. {\bf 241}  (2003), 79--92.


%\bibitem{ZarhinL} Yu.G. Zarhin,  {\em Homomorphisms of abelian varieties}.
%In: Y. Aubry, G. Lachaud (eds.)
%Arithmetic, Geometry and Coding Theory (AGCT 2003),
%S\'eminaires et Congr\'es {\bf 11} (2005), 189--215 .


\bibitem{ZarhinM} Yu.G. Zarhin,   {\em Endomorphism algebras of superelliptic Jacobians}. In:
 F. Bogomolov, Yu. Tschinkel (eds.) Geometric methods in Algebra and Number Theory,
Progress in Math. {\bf 235}, 339--362, Birkh\"auser, Boston Basel
Berlin, 2005.

% \bibitem{ZarhinMZ} Yu. G. Zarhin,   {\em Non-isogenous superelliptic jacobians}.  Math. Z.
%{\bf 253} (2006), 537--554.

\bibitem{ZarhinPisa} Yu. G. Zarhin,   {\em  Superelliptic jacobians}. In:  ``Diophantine Geometry" Proceedings (U.
Zannier, ed.),  Edizioni Della Normali, Pisa 2007, pp. 363-390.


  \bibitem{ZarhinMZ2} Yu.G. Zarhin, {\em Endomorphisms of superelliptic jacobians}. Math. Z.,
{\bf 261} (2009), 691--707, 709.

\end{thebibliography}
\end{document}